\documentclass[12pt,reqno]{amsart}
\usepackage[T1]{fontenc}
\usepackage[english]{babel}
\usepackage[utf8]{inputenc}
\usepackage[square,numbers]{natbib}
\usepackage{amssymb}
\usepackage{amsmath}
\usepackage{xcolor}
\usepackage{amsfonts}
\usepackage{amsthm}
\usepackage{csquotes}
\usepackage[a4paper,margin=2.5cm]{geometry}
\usepackage{graphicx} 
\usepackage{comment}
\usepackage{tikz}
\usepackage{subcaption}
\usepackage{enumitem}
\usepackage{float}
\usepackage{textgreek}

\usepackage[colorlinks]{hyperref}
\hypersetup{
 colorlinks = true,
 linkcolor = blue,
}
\setcitestyle{numbers}
\theoremstyle{plain}

\newcommand{\norm}[1]{\|#1\|}

\newcommand{\N}{\mathbb{N}}
\newcommand{\tr}{\operatorname{tr}}

\newcommand{\A}{\mathcal{A}}

\newcommand{\cumuf}{\kappa^{\operatorname{free}}}

\newcommand{\NC}{\operatorname{NC}}
\newcommand{\homo}{\operatorname{hom}}

\title[CLT]{Central limit theorem for $\epsilon$-independent products and higher-order tensors}
\author[G. Cébron and P.\ O.\ Santos \and P.\ Youssef]{%
        Guillaume Cébron \and
        Patrick Oliveira Santos \and 
        Pierre Youssef
        }

\address{Guillaume Cébron. Institut de Mathématiques de Toulouse; UMR5219; Université de Toulouse; CNRS; UPS, F-31062 Toulouse, France}
\email{guillaume.cebron@math.univ-toulouse.fr}
\address{Patrick Oliveira Santos. Division of Science, NYU Abu Dhabi, Abu Dhabi, UAE.}
\email{po2150@nyu.edu}
\address{Pierre Youssef. Division of Science, NYU Abu Dhabi, Abu Dhabi, UAE \& Courant Institute of Mathematical Sciences, New York University, New York, USA.}
\email{yp27@nyu.edu}

\begin{document}
\newtheorem{theorem}{Theorem}[section]

\newtheorem{corollary}[theorem]{Corollary}
\newtheorem{lemma}[theorem]{Lemma}
\newtheorem{conjecture}[theorem]{Conjecture}
\newtheorem{proposition}[theorem]{Proposition}

\theoremstyle{definition}
\newtheorem{example}[theorem]{Example}
\newtheorem{definition}[theorem]{Definition}
\newtheorem{remark}[theorem]{Remark}

\newtheorem*{model}{Model}

\maketitle
\begin{abstract}
We establish a central limit theorem (CLT) for families of products of $\epsilon$-independent random variables. Building on the approach in \cite{CSY}, we utilize graphon limits to encode the evolution of independence and characterize the limiting distribution. Our framework subsumes a wide class of dependency structures and includes, as a special case, a CLT for higher-order tensor products of free random variables. Our results generalize those in \cite{CSY}, extend the findings of \cite{lancien2024centrallimittheoremtensor}, which were limited to order-2 tensors, and recover as a special case the recent tensor-free CLT of \cite{nechita2025tensorfreeprobabilitytheory}, which was obtained through the development of a tensor analogue of free probability. In contrast, our approach is more direct and provides a unified and concise derivation of a more general CLT via graphon convergence.
\end{abstract}

\section{Introduction}
The Free Central Limit Theorem (CLT) serves as a cornerstone of free probability theory \cite{voiculescu1985symmetries}, \cite[Lecture 8]{nica2006lectures}. It parallels the classical CLT, replacing classical independence with free independence and the Gaussian limit with the semi-circle law. 
Recently, the authors \cite{CSY} unified these two by establishing a central limit theorem for $\epsilon$-independent variables. 
Let $(\A,\tau)$ be a noncommutative probability space. Based on the graph product of groups \cite{green1990graph}, the notion of $\epsilon$-independence was introduced in \cite{mlotkowski2004lambdafree} and elaborated in several subsequent works \cite{caspers2017graph,ebrahimifard2017,speicherjanusz2016mixture} (see also \cite{charlesworth2021matrix,charlesworth2024random} for matrix models). Given a loopless graph $g_n=([n],E)$ over $n$ vertices,  we say that subalgebras $\A_1,\ldots,\A_n \subset \A$ are $g_n$-independent if
\begin{enumerate}
    \item For every $(i,j) \in E$, the subalgebras $\A_i,\A_j$ are classical independent (in particular, they commute);
    \item For every $k \ge 1$, $i \in [n]^k$ such that for all $j_1<j_2$ with $i_{j_1}=i_{j_2}$, there exists $j_1<j_3<j_2$ with $(i_{j_1},i_{j_3}) \notin E$, then for any centered random variables $a_j \in \A_{i_j}$, $j \in [k]$, we have
    \begin{align*}
        \tau(a_1\cdots a_k)=0.
    \end{align*}
\end{enumerate}

Note that when the graph $g_n$ is complete (resp. edgeless), the above corresponds to classical independence (resp. freeness). 
The authors in \cite{CSY} used the theory of graphons to track the evolution of the graphs of independence. Recall that a graphon consists of a measurable bounded symmetric function $w:[0,1]^2 \to [0,1]$. Given a graphon $w$ and graphs $f=(V(f),E(f)), g=(V(g),E(g))$, we define
\begin{align*}
    &\rho(f,w):=\int_{[0,1]^{V(f)}} \prod_{(u,v) \in E} w(x_u,x_v) \, \text{d}x;\\
    &\rho(f,g)=\frac{1}{|V(g)|^{|V(f)|}}\sum_{\phi: V(f) \to V(g)}\prod_{(u,v)\in E(f)}\mathbf{1}_{(\phi(u),\phi(v))\in E(g)}.
\end{align*}
A sequence of simple graphs $g_n$ in $n$ nodes converges to a graphon $w$ if, for every simple finite graph $f=(V(f),E(f))$, we have
\begin{align*}
    \rho(f,g_n) \to \rho(f,w).
\end{align*}
We say that
a sequence of (self-adjoint) variables $(a_n)_{n\in \N}$ in $\A$  converges in distribution to $a\in \A$ if, for every $p\in \N$, $\tau(a_n^p) \to \tau(a^p)$ as $n\to \infty$. More generally, we define the convergence in $*$-moments of a sequence of not necessarily self-adjoint variables $(b_n)_{n\in \N}$ towards some variable $z$ if for any $k \ge 1$, $t \in \{1,*\}^k$, we have
\begin{align*}
    \tau(b_n^t) \to \tau(z^t),
\end{align*}
where
\begin{align*}
    b_n^t:=b_n^{t_1}\cdots b_n^{t_k}.
\end{align*}

The authors in \cite{CSY} established the following CLT for $\epsilon$-independent variables encompassing the classical and free central limit theorems. 
The combinatorial proof of those limit theorems latter deals with pair partitions. We denote by $P_2(p)$ the set of pair partitions of $[p]$ and for a pair partition $\pi\in P_2(p)$, we define its intersection graph $f_\pi$ whose vertex set consists of the blocks of $\pi$, and there exists an edge between two blocks if they cross.

\begin{theorem}[\cite{CSY}]\label{theorem: case L=1}
    Let $g_n$ be a graph over $[n]$, and $a\in \A$ be a self-adjoint random variable with mean $\lambda$ and variance $\sigma^2$. Assume that $g_n \to w$ in the graphon sense. Let $a_1,\ldots,a_n\in \A$ be $g_n$-independent and identically distributed random variables with common law $a$. Let
    \begin{align*}
        S_n=\frac{1}{\sigma\sqrt{n}}\sum_{k\in [n]}(a_k-\lambda).
    \end{align*}
    Then, $S_n$ converges in distribution to a measure $\mu_w$ that depends only on $w$, whose odd moments vanish and even moments are equal to
    \begin{align*}
        \int x^{2p}\,\text{d}\mu_w=\sum_{\pi \in P_2(2p)} \rho(f_\pi,w). 
    \end{align*}
\end{theorem}
The class of measures $\mu_w$ is quite vast, encompassing well-known distributions like $q$-Gaussians as well as many exotic laws (see \cite{CSY} for details).

In another direction motivated by the study of the limiting behavior of random quantum channels, \cite{lancien2024centrallimittheoremtensor} establishes a central limit theorem for tensor product random variables of the form $c_k=a_k\otimes a_k$, where $(a_k)_{k\in \N}$ is a free family of variables. 

\begin{theorem}[\cite{lancien2024centrallimittheoremtensor}]\label{theorem: L=2}
     Let $a \in \A$ be a self-adjoint random variable with mean $\lambda$ and variance $\sigma^2$, and denote  
    $$
    \delta^2:= \sigma^2(\sigma^2+2\lambda^2),\quad \alpha:=\frac{2\lambda^2}{\sigma^2+2\lambda^2} \in [0,1].
    $$
Given $(a_k)_{k\in \N}$ a sequence of free copies of $a$, the distribution $\mu_{S_n}$ of the normalized sum 
    \begin{align*}
        S_n:= \frac{1}{\delta \sqrt{n}}\sum_{k \in [n]}(a_k\otimes a_k-\lambda^2)
    \end{align*}
    converges in distribution as $n\to \infty$ to 
    \begin{equation*}
        \sqrt{\alpha}\left(\frac{1}{\sqrt{2}}s_1+\frac{1}{\sqrt{2}}s_2\right)+\sqrt{1-\alpha}\,s_3,
    \end{equation*}
    where $s_1,s_2,s_3$ are semi-circle random variables, $s_1,s_2$ are classical independent, both free independent from $s_3$.
\end{theorem}
The limiting law appearing in the above theorem is partially reminiscent of the laws uncovered in Theorem~\ref{theorem: case L=1}. For instance, the law of $\frac{1}{\sqrt{2}}s_1+\frac{1}{\sqrt{2}}s_2$ coincides with that of $\mu_w$ when the graphon $w$ is given by $ \mathbf{1}_{[0,1/2)\times [1/2,1]}+\mathbf{1}_{[1/2,1]\times[0,1/2)}$ (see \cite[Example~7.4]{CSY}). 
In view of this, it is natural to seek a connection between both results. 
Moreover, the analysis of higher-order tensors naturally generalizes the results for order-two tensors, models more intricate structures, and introduces a rich interplay between free and classical independence. Notably, tensors of random variables correspond to products of independent variables. Indeed, tensors of the form $a_k^{(1)}\otimes \ldots\otimes a_k^{(L)}$ coincide with the product of $L$ independent variables $b_j= 1\otimes\ldots \otimes a_{k}^{(j)}\otimes \ldots \otimes 1$, $1\leq j\leq L$. As such, a CLT for fixed-order tensors can be interpreted as a study of the free sum of products of independent variables. This realization provides a new perspective, allowing us to model these products and sums through mixtures of free and classical independence. 

To formalize this idea, we employ the concept of $\epsilon$-independence mentioned above, enabling us to consider a general setting beyond the fixed-order tensors, considering mixtures of free and independent variables, as well as mixtures of their sums and products. 
In \cite{CSY}, we demonstrated the utility of graphon theory in tracking the evolution of these mixtures and establishing the corresponding limit theorems. Here, we also make use of these techniques to analyze the sum of mixtures of products modeled by $\epsilon$-independence. More precisely, let $g_L$ be a graph over $[L]$, and $(\A_l)_{l\in [L]}\subset \A$ be $g_L$-independent subalgebras. For every $1\leq l\leq L$ and every $n\geq 1$,  let $g'_n$ be a graph over $[n]$ and consider variables $(a_k^{(l)})_{k\in [n]}\subset \A_l$ which are $g'_n$-independent. We aim to capture the limiting behavior, as $n\to \infty$, of 
$$
\frac{1}{\sqrt{n}}\sum_{k=1}^n \big(\prod_{l=1}^L a_k^{(l)}-\lambda^L\big).
$$
Theorem~\ref{theorem: case L=1} corresponds to the case $L=1$, while Theorem~\ref{theorem: L=2} corresponds to $L=2$ with $g_2$ being the complete graph and $g_n'$ being the edgeless graph. 
Our main theorem generalizes both results to cover any depth $L$ and any graphs of independence $g_L$ and $g_n'$.

\begin{theorem}\label{theorem: main theorem introduction}
  Let $g_L$ be a graph over $[L]$, $g'_n$ be a graph over $[n]$, and $a\in \A$ be a self-adjoint random variable with mean $\lambda$ and variance $\sigma^2$. Assume that $g'_n\to w$ in the graphon sense. Let $(\A_l)_{l\in [L]}$ be $g_L$-independent subalgebras and $(a_k^{(l)})_{k\in [n]}\subset \A_l$ be $g'_n$-independent and identically distributed random variables with common law $a$. Let
    \begin{align*}
        S_n=\frac{1}{\sqrt{n}}\sum_{k\in [n]}\big(\prod_{l=1}^L a_k^{(l)}-\lambda^L\big).
    \end{align*}
    Then, $S_n$ converges in $*$-distribution to an explicit measure $\mu_{g_L,w,\lambda,\sigma}$ that depends on $g_L,w,\lambda,\sigma$.
\end{theorem}

To improve readability and avoid technicalities, we do not explicitly state the limiting law $\mu_{g_L,w,\lambda,\sigma}$ here, referring instead to Section~\ref{sec: master theorem} for a precise formulation (see Theorem~\ref{theorem: precise main theorem}). When $L=1$, this law simplifies to $\mu_w$ and Theorem~\ref{theorem: main theorem introduction} recovers Theorem~\ref{theorem: case L=1}. Likewise, in the setting of Theorem~\ref{theorem: L=2}, the law $\mu_{g_L,w,\lambda,\sigma}$ properly normalized aligns with the one in Theorem~\ref{theorem: L=2}. Further details are provided in Section~\ref{sec: master theorem}. 
    
To illustrate the broad applicability of Theorem~\ref{theorem: main theorem introduction}, we consider the special case where $g_n'$ is the edgeless graph, leading to the following corollary. 

\begin{corollary}\label{corollary: tensor free case}
    Let $L\ge 1$ be fixed, and $g_L$ be a simple graph over $L$ vertices. Let $\A_1,\ldots,\A_L \subset \A$ be $g_L$-independent subalgebras and $a \in \A$ be a self-adjoint random variable with mean $\lambda$ and variance $\sigma^2$. For each $l\in [L]$, let $(a_n^{(l)})_{n \ge 1} \subset \A_l$ be free i.i.d copies of $a$, and define
    \begin{align*}
        S_n=\frac{1}{\sqrt{n}}\sum_{k \in [n]}\big(\prod_{l \in [L]}a_k^{(l)}-\lambda^L\big).
    \end{align*}
    Then, $S_n$ converges in $*$-moments to
    \begin{align*}
    \sum_{J \in \mathcal{P}^*(L)} \lambda^{|J^c|}\sigma^{|J|}s_{J},
    \end{align*}
    where $\mathcal{P}^*(L)$ denotes the collection of nonempty subsets of $[L]$. 
    Each variable $s_J$ is
    \begin{itemize}
        \item a semi-circle random variable if the subgraph of $g_L$ generated by $J$ is the complete graph;
        \item a circular random variable otherwise. 
    \end{itemize}
    Finally, the collection $(s_J)_{J\in \mathcal{P}^*(L)}$ is $h_L$-independent, where the set of vertices and edges of $h_L$ are given by
    \begin{align*}
        &V(h_L)=\mathcal{P}^*(L);\\
        &E(h_L)=\{(J_1,J_2): J_1 \times J_2 \subseteq E(g_L)\}.
    \end{align*}
\end{corollary}

Recall that a circular random variable $c$ has a law given by 
    $c=\frac{s+\sqrt{-1}s'}{\sqrt{2}}$, 
where $s,s'$ are free identically distributed semi-circle laws. Further specified to the setting of tensor products, i.e., taking $g_L$ to be the complete graph, we deduce the following corollary. 

\begin{corollary}\label{corollary: classic tensor free case}
    Let $L\geq 1$ be fixed, $\A_1,\ldots,\A_L$ classical independent subalgebras, and $a_1^{(l)},\ldots,a_n^{(l)} \in \A_l$ be self-adjoint free identically distributed random variables, with mean $\lambda$ and variance $\sigma^2$, for all $l \in [L]$. Then,
    \begin{align*}
        S_n=\frac{1}{\sqrt{n}}\sum_{k \in [n]}\big(\bigotimes_{l \in [L]}a_k^{(l)}-\lambda^L\big)
    \end{align*}
    converges in distribution towards $\sum_{J \in \mathcal{P}^*(L)}\lambda^{|J^c|}\sigma^{|J|}s_J$, 
    where the $s_J$'s are $h_L$-independent semi-circle laws with the edge set of $h_L$ given by
    \begin{align*}
        E(h_L):=\{(J_1,J_2)\in \mathcal{P}^*(L): J_1 \cap J_2=\varnothing\}.
    \end{align*}
\end{corollary}
It is easy to check that when $L=2$, the above recovers Theorem~\ref{theorem: L=2}, unveiling a law consisting of classical and free convolutions of semi-circle random variables. Moreover, note that if $\lambda=0$, then for any $L$, the above limiting law reduces to the semi-circle, extending a phenomenon observed in the case of order-two tensors in \cite{lancien2024centrallimittheoremtensor}.  
More generally, Corollary~\ref{corollary: classic tensor free case} recovers the recent tensor-free central limit theorem established in \cite[Theorem~1.2, Remark 9.3]{nechita2025tensorfreeprobabilitytheory}. The approach in \cite{nechita2025tensorfreeprobabilitytheory} consists of introducing the notion of tensor freeness, defining tensor-free cumulants, and establishing a corresponding moment-cumulant formula. In contrast, we deduce Corollary~\ref{corollary: classic tensor free case} as a special case of a more general framework for sums of products of $\epsilon$-independent variables, leveraging the theory of decorated graphons to formulate and prove a corresponding general limiting theorem (see below). Once the appropriate landscape is set up, the proof is reduced to a direct application of the moment method, which is carefully coupled with graphon convergence.

Theorem \ref{theorem: main theorem introduction} and its Corollaries are inscribed in a yet more general setting, which we elaborate in Section~\ref{sec: master theorem} (see Theorem~\ref{theorem: master theorem}). The general setting considers an independence structure captured by a grid graph over $[n]\times [L]$, while Theorem~\ref{theorem: main theorem introduction} deals with the special case of a grid graph given by the lexicographical product of a graph $g_L$ over $[L]$ and a graph $g_n'$ over $[n]$. We use the so-called decorated graphon theory from \cite{lovasz2022multigraph} to study the convergence of grid graphs, following a route similar to \cite[Theorem 5.1]{CSY}. We refer to Section~\ref{sec: preliminaries} for the precise definitions.

\section{Preliminaries}\label{sec: preliminaries}
In this section, we recall key notions from free probability and graphon theory. We denote by $P(k)$ the set of all partitions of $[k]$, by $P_2(k)$ the set of all pair partitions of $[k]$, by $\NC(k)$ of all noncrossing partition of $[k]$, and by $\NC_2(k)$ the set of all noncrossing pair partitions of $[k]$. More generally, for $T\subseteq P(k)$, we denote
\begin{align*}
    &T_{\le 2}=\{\pi\in T: |v| \le 2, \forall v\in \pi\};
    &T_{\ge 2}=\{\pi\in T: |v| \ge 2, \forall v\in \pi\};\\
    &T_2=T_{\le 2}\cap T_{\ge 2}.
\end{align*}
For a given partition $\pi\in P(p)$ and $r,s\in [p]$, we write $r \sim_\pi s$ if $r,s$ belong to the same block of $\pi$. 

For $n,k \ge 1$, and $i\in [n]^k$, we denote $\ker(i) \in P(k)$ the partition such that $i_k=i_l$ if and only if $k$ and $l$ belong to the same block of $\ker(i)$. 
A central tool in our analysis is the moment-cumulant formula for $\epsilon$-independence \cite{speicherjanusz2016mixture}. Given a simple graph $g_n=([n],E)$ in $n$ nodes, we denote $\NC(g_n,i)$ the set of $(g_n,i)$-noncrossing partitions, which are the partitions $\pi \in P(k)$ satisfying the following two properties
\begin{itemize}
    \item $\pi \le \ker(i)$, meaning that $\pi$ only groups elements that are already connected in $\ker (i)$;
    \item if there exist indices $p_1<p_2<q_1<q_2$ such that $i_{p_1} \nsim_\pi i_{p_2}$, $i_{p_1} \sim_\pi i_{q_1}$, $i_{p_2}\sim_\pi i_{q_2}$, then the edge $(i_{p_1},i_{p_2})$ must be present in $E$.
\end{itemize}  
Finally, we denote the free cumulants of $(a_1,\ldots,a_n)$ by $\cumuf_n(a_1,\ldots,a_n)$.

\begin{lemma}[\cite{speicherjanusz2016mixture}]\label{lemma: speicher}
    Let $n,k \ge 1$, $g_n=([n],E)$ be a simple graph in $n$ nodes and $i \in [n]^k$. Let $\A_1,\ldots,\A_n$ be $g_n$-independent subalgebras, and $a_j \in \A_{i_j}$, for $j \in [k]$. Then
    \begin{align*}
        \tau(a_1\cdots a_k)=\sum_{\pi \in \NC(g_n,i)} \prod_{v=\{r_1,\ldots,r_m\} \in \pi} \cumuf_m(a_{r_1},\ldots,a_{r_m}).
    \end{align*}
\end{lemma}

As mentioned at the end of the introduction, we will make use of the notion of decorated graphon \cite{lovasz2022multigraph}, which we will now formally introduce. 

Let $L\ge 1$, $\mathcal{P}^*(L)$ be the collection of nonempty subsets of $[L]$ and $\mathcal{P}(L)$ be the collection of subsets of $[L]$. We denote $d_L=|\mathcal{P}(L)|=2^L$. Let $\mathcal{M}_L$ be the space of $[0,1]$-valued $d_L\times d_L$ matrices. We equip this space with the inner product  and Hilbert-Schmidt norm
\begin{align*}
    \langle S,T\rangle=\tr(ST^*)=\sum_{I,J\in \mathcal{P}(L)}S_{IJ}T_{IJ}\, ;\quad 
    \norm{S}_{\operatorname{HS}}^2=\sum_{I,J\in \mathcal{P}(L)}S_{IJ}^2.
\end{align*}
An $\mathcal{M}_L$-decorated graphon $w$ is a map $w:[0,1]^2\to \mathcal{M}_L$ satisfying the symmetry condition $\norm{w(x,y)}_{\operatorname{HS}}=\norm{w(y,x)}_{\operatorname{HS}}$ for every $x,y\in [0,1]$. An $\mathcal{M}_L$-decorated graph is a pair $G=(g,\beta)$, where:  
\begin{itemize}
    \item $g=(V(g),E(g))$ is a graph;
    \item $\beta: E(g)\to \mathcal{M}_L$ is a map that assigns a matrix to each edge. 
\end{itemize}
Given such a decorated graph, we can define the associated decorated graphon by
\begin{align*}
    w_G(x,y)=\beta(i,j), \quad \text{for }(x,y)\in \left[\frac{i-1}{n},\frac{i}{n}\right)\times \left[\frac{j-1}{n},\frac{j}{n}\right).
\end{align*}
Given an $\mathcal{M}_L$-decorated graph $F=(f,\beta)$ and an $\mathcal{M}_L$-decorated graphon $w$, we define the homomorphism density
\begin{align*}
    \rho(F,w)=\int_{[0,1]^{V(f)}}\prod_{(u,v)\in E(f)}\langle \beta(u,v),w(x_u,x_v)\rangle\,\text{d}x. 
\end{align*}
When the graphon $w$ is induced by a decorated graph $G$ i.e. $w=w_G$, this simplifies to 
\begin{align*}
    &\rho(F,G):=\rho(F,w_G)=\frac{\hom(F,G)}{|V(g)|^{|V(f)|}};\\
    &\hom(F,G)=\sum_{\phi:V(f)\to V(g)}\prod_{e\in E(f)}\langle \beta(e),\gamma(\phi(e))\rangle,
\end{align*}
where $\phi(e)=(\phi(u),\phi(v))$, for $e=(u,v)$. 
Similarly, the injective homomorphism density is defined as
\begin{align*}
    &\homo_{\operatorname{inj}}(F,G)=\sum_{\phi:V(f)\hookrightarrow V(g)}\prod_{e\in E(f)}\langle \beta(e),\gamma(\phi(e))\rangle;\\
    &\rho_{\operatorname{inj}}(F,G)=\frac{\homo_{\operatorname{inj}}(F,G)}{|V(g)|^{|V(f)|}},
\end{align*}
where $\phi: V(f) \hookrightarrow V(g)$ denotes an injective map.

A sequence of $\mathcal{M}_L$-decorated graphons $(w_n)_{n\geq 1}$ is said to converge to an $\mathcal{M}_L$-decorated graphon $w$, denoted by $w_n\to w$, if $\rho(F,w_n)\to \rho(F,w)$ for every $\mathcal{M}_L$-decorated graph $F$ (see \cite[Theorem 3.7]{lovasz2022multigraph} for more details). This convergence can be expressed using the cut distance $\delta_{\Box}$ defined as:
\begin{align*}
    &\|w\|_{\Box}:=\sup_{f,g:[0,1]\to [0,1]}\left\|\int_{[0,1]^2}w(x,y) f(x)g(y)\text{d}x\text{d}y\right\|_{\operatorname{HS}};\\
    &\delta_{\Box}(w_1,w_2):=\inf_{\phi}\norm{w_1-w_2\circ \phi}_{\Box},
\end{align*}
where the infimum is over measure-preserving transformations $\phi:[0,1]\to [0,1]$ satisfying $w\circ\phi (x,y):=w(\phi(x),\phi(y))$.
 This cut distance provides a criterion for the convergence $w_n\to w$. The proof of continuity for the mapping  $w\mapsto \rho(F,w)$ under $\delta_{\Box}$ is detailed in \cite[Section 2]{lovasz2022multigraph}.

In our setting, we utilize decorated graphons in a specific way. Given a graph $f=(V(f),E(f))$ and $J\in \mathcal{P}^*(L)^{V(f)}$, we define its $\mathcal{M}_L$-decorated graph $F_f^J=(f,\beta_J)$ where the decoration map is given by 
\begin{align*}
    \beta_J(u,v)=\mathbf{P}_{J_u,J_v},\; \text{for } (u,v)\in E(f),
\end{align*}
where $\mathbf{P}_{I,J}$ denotes the canonical basis element in $\mathcal{M}_L$. 
For an intersection graph $f_\pi$ of a pair partition $\pi\in P_2(p)$ and a word $\alpha\in \{1,*\}^p$, we define $F_{f_\pi}^{J,\alpha}=(\Tilde{f}_\pi,\beta_J)$ where $\Tilde{f}_\pi$ extends $f_\pi$ by adding an extra vertex $\xi$ such that
\begin{subequations}
\label{eq: decorated intersection graph}
    \begin{align}
    &V(\Tilde{f}_\pi)=V(f_\pi)\sqcup \{\xi\}\\
    &E(\Tilde{f}_\pi)=E(f_\pi)\sqcup \{(\{r,s\},\xi): \{r,s\}\in \pi, \alpha_r=\alpha_s\}\\
    &\beta_J(u,\xi)=\mathbf{P}_{J_u\varnothing };\quad \beta_J(\xi,u)=\mathbf{P}_{ \varnothing J_u}.
\end{align}
\end{subequations}

For a graph $g_n=([n]\times [L],E(g_n))$, we define an $\mathcal{M}_L$-decorated graph $G_{g_n}=(V(G_{g_n}),E(G_{g_n}),\gamma_n)$ where $(V(G_{g_n}),E(G_{g_n}))$ corresponds to the complete graph $K_n$ over $n$ vertices. The decoration map $\gamma_n$ is given for $u\ne v\in [n]$ by
\begin{subequations}
\label{eq: decorated compressed grid}
\begin{align}
    &(\gamma_n(u,v))_{J_1J_2}=\prod_{(l_1,l_2)\in J_1\times J_2}\mathbf{1}_{((u,l_1),(v,l_2))\in E(g_n)}\\
    &(\gamma_n(u,v))_{J\varnothing}=\prod_{\substack{(l_1,l_2)\in J\times J\\ l_1\ne l_2}}\mathbf{1}_{((u,l_1),(u,l_2))\in E(g_n)}\\
    &(\gamma_n(u,v))_{\varnothing J}=\prod_{\substack{(l_1,l_2)\in J\times J\\ l_1\ne l_2}}\mathbf{1}_{((v,l_1),(v,l_2))\in E(g_n)}\\
    &(\gamma_n(u,v))_{\varnothing \varnothing}=1,
\end{align}
\end{subequations}
for $J_1,J_2,J\in \mathcal{P}^*(L)$.
Finally note that $\norm{\gamma_n(u,v)}_{\operatorname{HS}}=\norm{\gamma_n(v,u)}_{\operatorname{HS}}$. The following lemma allow us to compute $\rho_{\operatorname{inj}}(F^{J,\alpha}_{f_\pi},G_{g_n})$.
\begin{lemma}\label{lemma: rho computation}
    Let $\pi\in P_2(p)$, $\alpha \in \{1,*\}^p$, $J\in [\mathcal{P}^*(L)]^{V(f_\pi)}$ and $g_n=([n]\times [L],E(g_n))$. Let $G_{g_n}=(K_n,\gamma_n)$. Then
    \begin{align*}
    \rho_{\operatorname{inj}}(F^{J,\alpha}_{f_\pi},G_{g_n})&=\left(\frac{n-p/2}{n}\right)\frac{1}{n^{p/2}}\sum_{\phi: V(f_\pi)\hookrightarrow [n]}\prod_{\substack{u=\{r,s\}\in V(f_\pi)\\ \alpha_r=\alpha_s}}\left(\prod_{\substack{(l_1,l_2)\in J_u^2\\l_1\ne l_2}}\mathbf{1}_{(\phi(u),l_1),(\phi(u),l_2))\in E(g_n)}\right) \\
    &\times\prod_{(u,v)\in E(f_\pi)}\left(\prod_{(l_1,l_2)\in J_u\times J_v}\mathbf{1}_{(\phi(u),l_1),\phi(v),l_2))\in E(g_n)}\right).
    \end{align*}
\end{lemma}
\begin{proof}
    We write
    \begin{align*}
        \rho_{\operatorname{inj}}(F^{J,\alpha}_{f_\pi},G_{g_n})=\frac{1}{n^{1+p/2}}\sum_{\phi: V(f_\pi)\sqcup \{\xi\}\hookrightarrow [n]}\prod_{(u,v)\in E(\Tilde{f}_\pi)}\langle \beta_J(u,v),\gamma_n(\phi(u),\phi(v))\rangle.
    \end{align*}
    We first consider an edge $(u,\xi)\in E(\Tilde{f}_\pi)$. By definition, we have $u=\{r,s\}$ with $\alpha_r=\alpha_s$. We have
    \begin{align*}
        \langle \beta_J(u,\xi),\gamma(\phi(u),\phi(\xi))\rangle&= (\gamma_n(\phi(u),\phi(\xi)))_{J\varnothing}\\
        &=\prod_{\substack{(l_1,l_2)\in J_u^2\\ l_1\ne l_2}}\mathbf{1}_{((\phi(u),l_1),(\phi(u),l_2))\in E(g_n)}.
    \end{align*}
    In particular, the inner product does not depend on $\phi(\xi)$. Summing over $\phi(\xi)\in [n]\setminus\{\phi(u):u\in V(f_\pi)\}$, we have
    \begin{align*}
         \rho(F^{J,\alpha}_{f_\pi},\gamma_n)&=\left(\frac{n-p/2}{n}\right)\frac{1}{n^{p/2}}\sum_{\phi: V(f_\pi)\hookrightarrow [n]}\prod_{\substack{u=\{r,s\}\in V(f_\pi)\\ \alpha_r=\alpha_s}}\left(\prod_{\substack{(l_1,l_2)\in J_u^2\\l_1\ne l_2}}\mathbf{1}_{(\phi(u),l_1),(\phi(u),l_2))\in E(g_n)}\right)\\
         &\times\prod_{(u,v)\in E(f_\pi)}\langle \beta_J(u,v),\gamma_n(\phi(u),\phi(v))\rangle.
    \end{align*}
    For an edge $(u,v)\in E(f_\pi)$, we have
    \begin{align*}
        \langle \beta_J(u,v),\gamma(\phi(u),\phi(v))\rangle&= (\gamma_n(\phi(u),\phi(v)))_{J_uJ_v}\\
        &=\prod_{(l_1,l_2)\in J_u\times J_v}\mathbf{1}_{((\phi(u),l_1),(\phi(v),l_2))\in E(g_n)}.
    \end{align*}
    The result follows.
\end{proof}

\section{General theorem via decorated graphons}\label{sec: master theorem}

We now present the main result of this section, which provides a general limit theorem for sums of products indexed by grid graphs. The theorem applies to $g_n$-independent random variables and describes their asymptotic behavior through decorated graphon convergence. In this section, we will demonstrate how it leads to the main result of the paper, as stated in the introduction, along with its corollaries. The proof of the theorem will be presented in the following section.

We say that a random variable $a$ is normalized if $\tau(a)=0$ and $\tau(a^2)=1$. For $J\in \mathcal{P}^*(L)^p$, we denote
\begin{align*}
    P_2^J(p):=\{\pi\in P_2(p): \forall \{r,s\}\in \pi, J_r=J_s\}=\{\pi\in P_2(p): \pi \le \ker(J)\}.
\end{align*}
For $\pi \in P_2^J(p)$, we denote $J_u:=J_r=J_s$, for $u=\{r,s\}\in \pi$. 
\begin{theorem}\label{theorem: master theorem}
    Let $g_n=([n]\times[L],E(g_n))$ be a grid graph such that $G_{g_n}$ converges in the decorated-graphon sense to a decorated graphon $w$. Let $a\in \A$ be a self-adjoint normalized random variable and consider a family of $g_n$-independent and identically distributed random variables $(a_k^{(l)})_{k\in [n],l\in [L]}$ with common law $a$. Define
    \begin{align*}
        S_n^{(J)}=\frac{1}{\sqrt{n}}\sum_{k\in [n]}b_k^{(J)};\quad b_k^{(J)}=\prod_{l\in J}a_k^{(l)},
    \end{align*}
    for $J\in \mathcal{P}^*(L)$. 
    Then, the family $(S_n^{(I)})_{I\in \mathcal{P}^*(L)}$ converges in distribution to a family $(s_I)_{I\in \mathcal{P}^*(L)}$, whose joint distribution depends only on $w$. Moreover, for any $p\ge 1$, $J\in \mathcal{P}^*(L)^p$ and $\alpha\in \{1,*\}^p$, we have
    \begin{align*}
        \tau(s_{J_1}^{\alpha_1}\cdots s_{J_p}^{\alpha_p})=\sum_{\pi \in P_2^{J}(p)}\rho(F_{f_\pi}^{J,\alpha},w).
    \end{align*}
\end{theorem}

The main theorem presented in the introduction focuses on a specific case where the grid graph is defined as the lexicographic product of graphs. In this setting, it establishes the limiting distribution of $S_n$. In contrast, the more general theorem stated above extends to the joint distribution of all partial sums. 

Let $g_L$ be a graph over $[L]$, and let $g'_n$ be a graph over $[n]$. We consider the lexicographical product graph $g_{n,L}=g'_n\cdot g_L$, whose vertex set is $[n]\times [L]$ and edge set
\begin{align*}
    E(g_{n,L})=\{((i,l),(j,l)): (i,j)\in E(g'_n))\}\sqcup\{((i,l_1),(j,l_2)): (l_1,l_2)\in E(g_L)\}.
\end{align*}
In this construction, each layer $l$ (corresponding to a fixed value in $[L]$) inherits the edges of $g_n'$, while edges between layers follow the structure of $g_L$ (see Figure \ref{fig: lexicographic product} for an example $n=L=3$).
\begin{figure}[H]
     \centering
     \begin{subfigure}[b]{0.3\textwidth}
         \centering
         \begin{tikzpicture}[scale=0.50]
            \node at (0,0) (A) {};
            \fill[black] (A) circle(1.5pt);
            \node at (0,1) (B) {};
            \fill[black] (B) circle(1.5pt);
            \node at (0,2) (C) {};
            \fill[black] (C) circle(1.5pt);

            \draw[-] (0,0) edge (0,1);
            \draw[-] (0,1) edge (0,2);
        \end{tikzpicture}
        \caption{Line graph $g_L$.}
    \end{subfigure}
     \hfill
     \begin{subfigure}[b]{0.3\textwidth}
         \centering
         \begin{tikzpicture}[scale=0.50]
            \node at (0,0) (A) {};
            \fill[black] (A) circle(1.5pt);
            \node at (1,0) (B) {};
            \fill[black] (B) circle(1.5pt);
            \node at (2,0) (C) {};
            \fill[black] (C) circle(1.5pt);

            \draw[-] (1,0) edge (2,0);
            \draw[-] (0,0) edge [bend left=30] (2,0);
        \end{tikzpicture}
         \caption{Graph $g'_n$.}
     \end{subfigure}
     \hfill
     \begin{subfigure}[b]{0.3\textwidth}
         \centering
         \begin{tikzpicture}[scale=0.50]
            \node at (0,0) (A) {};
            \fill[black] (A) circle(1.5pt);
            \node at (0,1) (B) {};
            \fill[black] (B) circle(1.5pt);
            \node at (0,2) (C) {};
            \fill[black] (C) circle(1.5pt);

            \node at (1,0) (A) {};
            \fill[black] (A) circle(1.5pt);
            \node at (1,1) (B) {};
            \fill[black] (B) circle(1.5pt);
            \node at (1,2) (C) {};
            \fill[black] (C) circle(1.5pt);

            \node at (2,0) (A) {};
            \fill[black] (A) circle(1.5pt);
            \node at (2,1) (B) {};
            \fill[black] (B) circle(1.5pt);
            \node at (2,2) (C) {};
            \fill[black] (C) circle(1.5pt);

            \foreach \x in {0,1,2}
            \foreach \y in {0,1,2}
              {
                \draw[-] (\x,0) edge (\y,1);
                \draw[-] (\x,1) edge (\y,2);
              }
            \draw[-] (1,0) edge (2,0);
            \draw[-] (0,0) edge [bend right=30] (2,0);        

            \draw[-] (1,1) edge (2,1);
            \draw[-] (0,1) edge [bend left=20] (2,1);   

            \draw[-] (1,2) edge (2,2);
            \draw[-] (0,2) edge [bend left=30] (2,2);   
        \end{tikzpicture}
         \caption{Product $g'_n\cdot g_L$.}
     \end{subfigure}
        \caption{Example of lexicographical product.}
        \label{fig: lexicographic product}
\end{figure}
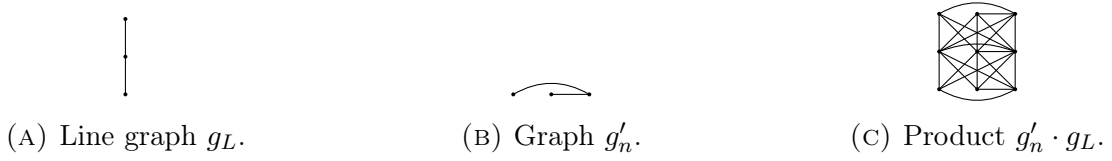
In this case, the variables $(a_k^{(l)})_{k\in [n]}\subset \A_l$ are $g'_n$-independent while the subalgebras $(\A_l)_{l\in [L]}$ are $g_L$-independent. Let $J\in \mathcal{P}^*(L)^p$ and $\alpha\in \{1,*\}^p$, and define $P^{J,\alpha}(p)$ as the subset of partitions $\pi\in P^J_2(p)$  that satisfy the following conditions: 
\begin{enumerate}
    \item For every $(u,v)\in E(f_\pi)$, we have $(l_1,l_2)\in E(g_L)$ for all $(l_1,l_2)\in J_u\times J_v$ with $l_1\ne l_2$;
    \item For every block $u=\{r,s\}\in \pi$ such that $\alpha_r=\alpha_s$, the restriction $g_L|_{J_u}$ is a complete graph, meaning that $(l_1,l_2)\in E(g_L)$ for all distinct $l_1,l_2\in J_u$.
\end{enumerate}
Specifying to this example, Theorem \ref{theorem: master theorem} yields the following.
\begin{corollary}\label{corollary: lexicographic joint law}
    Let $g_L$ be a graph over $[L]$, $g'_n$ be a graph over $[n]$, and $a\in \A$ be a normalized random variable. Assume that $g'_n$ converges to a graphon $w$ in the graphon sense. Let $(\A_l)_{l\in [L]}$ be $g_L$-independent subalgebras and $(a_k^{(l)})_{k\in [n]}\subset \A_l$ be self-adjoint $g'_n$-independent and identically distributed random variables with common law $a$. Define
    \begin{align*}
        S_n^{(J)}=\frac{1}{\sqrt{n}}\sum_{k\in [n]}b_k^{(J)},\quad \text{where } b_k^{(J)}=\prod_{l\in J}a_k^{(l)}.
    \end{align*}
    Then, the family $(S_n^{(I)})_{I\in \mathcal{P}^*(L)}$ converges in distribution to a family $(s_I)_{I\in \mathcal{P}^*(L)}$, whose joint law $\nu_{g_L,w}$ depends only on $g_L$ and $w$. 
    
    Moreover, for any $p\ge 1$, $J\in \mathcal{P}^*(L)^p$ and $\alpha\in \{1,*\}^p$, we have
    \begin{equation}\label{eq: def joint sJ}
        \tau(s_{J_1}^{\alpha_1}\cdots s_{J_p}^{\alpha_p})=\sum_{\pi \in P_2^{J,\alpha}(p)}\int_{[0,1]^{V(f_\pi)}}\prod_{(u,v)\in E(f_\pi)} w^*_{J_uJ_v}(x_u,x_v)\,\text{d}x,
    \end{equation}
    where
    \begin{align*}
        w^*_{J_1J_2}(x,y)=\begin{cases}
            w(x,y);& \text{ if } J_1\cap J_2 \ne \varnothing;\\
            1;& \text{ otherwise}.
        \end{cases}
    \end{align*}
\end{corollary}
\begin{proof}
Denote $g_n=g'_n\cdot g_L$. Given $J_1, J_2 \in \mathcal{P}^*(L)$ we have 
\begin{align*}
    \prod_{(l_1,l_2)\in J_1\times J_2} \mathbf{1}_{((i,l_1),(j,l_2))\in E(g_n)}&=\prod_{l\in J_1\cap J_2} \mathbf{1}_{((i,l),(j,l))\in E(g_n)}\prod_{\substack{(l_1,l_2)\in J_1\times J_2\\l_1\ne l_2}} \mathbf{1}_{((i,l_1),(j,l_2))\in E(g_n)}\\
    &=\prod_{l\in J_1\cap J_2}\mathbf{1}_{(i,j)\in E(g_n')}\prod_{\substack{(l_1,l_2)\in J_1\times J_2\\l_1\ne l_2}}\mathbf{1}_{(l_1,l_2)\in E(g_L)}.
\end{align*}
In view of this, and since $g_n'$ converges to a graphon $w$, the decorated graph $G_{g_n}=(K_n,\gamma_n)$ \eqref{eq: decorated compressed grid} converges to a decorated graphon $\gamma$ defined by
\begin{align*}
    (\gamma(x,y))_{J_1,J_2}= w^*_{J_1J_2}(x,y)\prod_{\substack{(l_1,l_2)\in J_1\times J_2\\l_1\ne l_2}}\mathbf{1}_{(l_1,l_2)\in E(g_L)},
\end{align*}
for $J_1,J_2\in \mathcal{P}^*(L)$ and
\begin{align*}
    (\gamma(x,y))_{\varnothing J}=(\gamma(x,y))_{J\varnothing }= \prod_{\substack{(l_1,l_2)\in J\times J\\l_1\ne l_2}}\mathbf{1}_{(l_1,l_2)\in E(g_L)},
\end{align*}
for $J\in \mathcal{P}^*(L)$.
Let $p\ge 1$, $J\in \mathcal{P}^*(L)^{p/2}$, $\alpha\in \{1,*\}^p$, and consider the decorated graphon $F_{f_\pi}^{J,\alpha}=(\Tilde{f}_\pi,\beta_J)$ defined in \eqref{eq: decorated intersection graph}. Note that 
\begin{align*}
\rho(F_{f_\pi}^{J,\alpha},\gamma)&=\int_{[0,1]}\int_{[0,1]^{V(f)}}\prod_{e=(u,v)\in E(\Tilde{f}_\pi)}\langle \beta(e),\gamma(x_u,x_v)\rangle\,\text{d}x\text{d}x_\xi\\
&=:\int_{[0,1]}\int_{[0,1]^{V(f)}} y_\pi(x)\,\text{d}x\text{d}x_\xi .
\end{align*} 
We write
\begin{align*}
    y_\pi(x)&=\prod_{e=(u,v)\in E(\Tilde{f}_\pi)}\langle \beta(e),\gamma(x_u,x_v)\rangle\\
    &=\prod_{\substack{v=\{r,s\}\in V(f_\pi)\\ \alpha_r=\alpha_s}}(\gamma(x_v,x_\xi))_{J_v \varnothing}\prod_{(u,v)\in E(f_\pi)}(\gamma(x_u,x_v))_{J_uJ_v}.
\end{align*}
By definition of $\gamma$, we deduce that
\begin{align*}
    &\frac{y_\pi(x)}{\left(\prod_{(u,v)\in E(f_\pi)} w^*_{J_uJ_v}(x_u,x_v)\right)}=\\
    &\left(\prod_{(u,v)\in E(f_\pi)} \prod_{\substack{(l_1,l_2)\in J_u\times J_v\\l_1\ne l_2}}\mathbf{1}_{(l_1,l_2)\in E(g_L)}\right)\left(\prod_{\substack{u=\{r,s\}\in V(f_\pi)\\ \alpha_r=\alpha_s}}\prod_{\substack{(l_1,l_2)\in J_u\times J_u\\l_1\ne l_2}}\mathbf{1}_{(l_1,l_2)\in E(g_L)}\right).
\end{align*}
By definition, the above product vanishes unless $\pi \in P_2^{J,\alpha}(p)$ in which case we have 
$$
\rho(F_{f_\pi}^{J,\alpha},\gamma) =\int_{[0,1]^{V(f)}}\prod_{(u,v)\in E(f_\pi)} 
w^*_{J_uJ_v}(x_u,y_v)\,\text{d}x.
$$
It remains to apply Theorem~\ref{theorem: master theorem} to finish the proof. 
\end{proof}

With this in hand, we can now state and prove an explicit version of Theorem~\ref{theorem: main theorem introduction}.

\begin{theorem}[Expanded version of Theorem~\ref{theorem: main theorem introduction}]\label{theorem: precise main theorem}
 Let $g_L$ be a graph over $[L]$, $g'_n$ be a graph over $[n]$, and $a\in \A$ be a self-adjoint random variable with mean $\lambda$ and variance $\sigma^2$. Assume that $g'_n\to w$ in the graphon sense. Let $(\A_l)_{l\in [L]}$ be $g_L$-independent subalgebras and $(a_k^{(l)})_{k\in [n]}\subset \A_l$ be $g'_n$-independent and identically distributed random variables with common law $a$. Let
    \begin{align*}
        S_n=\frac{1}{\sqrt{n}}\left\{\sum_{k\in [n]}\left(\prod_{l=1}^L a_k^{(l)}-\lambda^L\right)\right\}.
    \end{align*}
    Then, $S_n$ converges in $*$-distribution to $\sum_{J\in \mathcal{P}^*(L)}\sigma^{|J|}\lambda^{|J^c|} s_J$, where the family $(s_I)_{I\in \mathcal{P}^*(L)}$ has the joint law defined in \eqref{eq: def joint sJ}. 
\end{theorem}
\begin{proof}
For every $k\in [n]$ and $l\in [L]$, denote $z_k^{(l)}=\frac{a_k^{(l)}-\lambda}{\sigma}$. Then, we can write
\begin{align*}
    \prod_{l\in [L]}a_k^{(l)}-\lambda^L&=\prod_{l\in [L]}\left(\sigma z_k^{(l)}+\lambda\right)-\lambda^L\\
    &=\sum_{J\in \mathcal{P}^*(L)}\sigma^{|J|}\lambda^{|J^c|}\prod_{l\in J}z_k^{(l)}.
\end{align*}
Therefore, we deduce that 
\begin{align*}
    S_n&=\frac{1}{\sqrt{n}}\sum_{k\in [n]}\prod_{l\in [L]}a_k^{(l)}-\lambda^L
    =\sum_{J\in \mathcal{P}^*(L)}\sigma^{|J|}\lambda^{|J^c|}S_n^{(J)},
\end{align*}
where 
\begin{align*}
    S_n^{(J)}=\frac{1}{\sqrt{n}}\sum_{k\in [n]}\prod_{l\in J}z_k^{(l)}.
\end{align*}
Following Corollary~\ref{corollary: lexicographic joint law}, we deduce that $S_n$ converges to $\sum_{J\in \mathcal{P}^*(L)}\sigma^{|J|}\lambda^{|J^c|} s_J$, where the family $(s_I)_{I\in \mathcal{P}^*(L)}$ has the joint law  stated in Corollary~\ref{corollary: lexicographic joint law}. 
\end{proof}

We are now ready to show how Theorem~\ref{theorem: precise main theorem} implies Theorems~\ref{theorem: case L=1} and \ref{theorem: L=2} and Corollaries~\ref{corollary: tensor free case} and \ref{corollary: classic tensor free case}.

\begin{proof}[Proof of Theorem~\ref{theorem: case L=1}]
Applying Theorem~\ref{theorem: precise main theorem} with $L=1$, 
we only have $s=s_{\{1\}}$, whose law is given by
    \begin{align*}
        \tau(s^p)=\sum_{\pi\in P_2^{J,\alpha}(p)}\int_{[0,1]^{V(f_\pi)}}\prod_{(u,v)\in E(f_\pi)} w^*_{J_uJ_v}(x_u,x_v)\,\text{d}x,
    \end{align*}
    where $J=(\{1\},\ldots,\{1\})$ and $\alpha=(1,\ldots,1)$. Therefore, $P_2^{J,\alpha}=P_2(p)$, $w^*(x,y)=w(x,y)$ and Theorem~\ref{theorem: case L=1} follows. 
\end{proof}

\begin{proof}[Proof of Corollary \ref{corollary: tensor free case}]
    We first note that since $(a_k^{(l)})_{k\in [n]}\subset \A_l$ are free identically distributed, we have that the graph $g_n'$ is edgeless. Hence, its limiting graphon is equal to $w(x,y)=0$ for all $x\ne y$. 
    Applying Theorem~\ref{theorem: precise main theorem}, we deduce that  $S_n$ converges in $*$-moments to
    $\sum_{J \in \mathcal{P}^*(L)} \lambda^{|J^c|}\sigma^{|J|}s_{J}$ with the joint law of $(s_J)_{J\in \mathcal{P}^*(L)}$ given by 
     \begin{align*}
        \tau(s_{J_1}^{\alpha_1}\cdots s_{J_p}^{\alpha_p})=\sum_{\pi\in P_2^{J,\alpha }(p)}  \int_{[0,1]^{V(f_\pi)}}\prod_{(u,v)\in E(f_\pi)} w^*_{J_uJ_v}(x_u,x_v)\,\text{d}x=\sum_{\pi\in P_2^{J,\alpha }(p)}\prod_{(u,v)\in E(f_\pi)} \mathbf{1}_{J_u\cap J_v=\varnothing},
    \end{align*}
    for any $p\ge 1$, $J\in \mathcal{P}^*(L)^p$, and $\alpha\in \{1,*\}^p$. 
      Let $\overline{P}_2^{J,\alpha}(p)$ be those partitions $\pi\in P_2^{J,\alpha}(p)$ such that $
        \prod_{(u,v)\in E(f_\pi)} \mathbf{1}_{J_u\cap J_v=\varnothing}=1$. 
        It follows from the definition that $\overline{P}_2^{J,\alpha}(p) \subseteq \NC_2(h_L,J)$. Note that since $g_L$ has no loops, the condition $(J_1,J_2)\in E(h_L)$ implies that $J_1\cap J_2=\varnothing$. Therefore, following the definition of $\overline{P}_2^{J,\alpha}(p)$, a partition $\pi \in \NC_2(h_L,J)$ belongs to $\overline{P}_2^{J,\alpha}(p)$ if and only if 
        $$
        \prod_{u=\{r,s\}\in \pi} \mathbf{1}_{\{\alpha_r=\alpha_s\, \Rightarrow\, g_L|_{J_u}\text{ is the complete graph}\}}=1.
        $$
Thus, we deduce that 
$$
 \tau(s_{J_1}^{\alpha_1}\cdots s_{J_p}^{\alpha_p})=\sum_{\pi\in \NC_2(h_L,J)} \prod_{u=\{r,s\}\in \pi} \mathbf{1}_{\{\alpha_r=\alpha_s\, \Rightarrow\, g_L|_{J_u}\text{ is the complete graph}\}}.
$$
Let $I\in \mathcal{P}^*(L)$ be such that $g_L|_{I}$ is the complete graph. Then setting $J=(I,\ldots, I)$ and using that $h_L$ has no loops, we deduce that for any  $\alpha\in \{1,*\}^p$ we have
$$
 \tau(s_{I}^{\alpha_1}\cdots s_{I}^{\alpha_p})= |\NC_2(p)|,
$$
which implies that $s_I$ is a semi-circle random variable. Similarly, one can check that if $g_L|_{I}$ is not the complete graph, then $s_I$ is a circular random variable. In view of this, we see that for any $p\ge 1$, $J\in \mathcal{P}^*(L)^p$, $\alpha\in \{1,*\}^p$ and $\pi\in \NC_2(h_L,J)$, we have 
    \begin{align*}
       \prod_{u=\{r,s\}\in \pi} \mathbf{1}_{\{\alpha_r=\alpha_s\, \Rightarrow\, g_L|_{J_u}\text{ is the complete graph}\}}   =\prod_{u=\{r,s\}\in \pi}\cumuf_2(s_{J_u}^{\alpha_r},s_{J_u}^{\alpha_s}).
    \end{align*}
Putting together the above and using that only the second cumulants of circular and semi-circle variables are nonvanishing, we get that 
$$
 \tau(s_{J_1}^{\alpha_1}\cdots s_{J_p}^{\alpha_p})=\sum_{\pi\in \NC(h_L,J)}\cumuf_\pi(s_{J_1}^{\alpha_1},\ldots,s_{J_p}^{\alpha_p}),
$$
 for any $p\ge 1$, $J\in \mathcal{P}^*(L)^p$, and $\alpha\in \{1,*\}^p$. It remains to use Speicher's lemma \ref{lemma: speicher} to deduce that the collection $(s_J)_{J\in \mathcal{P}^*(L)}$ is $h_L$-independent and finish the proof. 
\end{proof}

\begin{proof}[Proof of Corollary~\ref{corollary: classic tensor free case}]
    Apply Corollary~\ref{corollary: tensor free case} with $g_L$ being the complete graph to deduce that $S_n$ converges to $\sum_{J \in \mathcal{P}^*(L)} \lambda^{|J^c|}\sigma^{|J|}s_{J}$. Since all restrictions of $g_L$ are complete graphs, we deduce that each $s_J$ is a semi-circle variable. Moreover, since $g_L$ is complete but does not have loops, we deduce the edge set of $h_L$ as stated. 
\end{proof}

\begin{proof}[Proof of Theorem~\ref{theorem: L=2}]
Let $\delta$ as in Theorem~\ref{theorem: L=2} and $S_n$ normalized accordingly. We apply Corollary~\ref{corollary: classic tensor free case} with $L=2$ to deduce that $S_n$ converges to 
$$
\frac{\sigma^2}{\delta} s_{\{1,2\}} + \frac{\sigma \lambda}{\delta} s_{\{1\}} + \frac{\sigma^2}{\delta} s_{\{2\}},
$$
with the variables $s_{\{1,2\}}, s_{\{1\}}$ and $s_{\{2\}}$ being semi-circles. Moreover, since  $\{1,2\}\cap \{1\}\neq \emptyset$ and $\{1,2\}\cap \{2\}\neq \emptyset$ while $\{1\}\cap \{2\}= \emptyset$, we deduce that $s_{\{1\}}$ and $s_{\{2\}}$ are classically independent while $s_{\{1,2\}}$ is free from $(s_{\{1\}}, s_{\{2\}})$. Theorem~\ref{theorem: L=2} follows. 
\end{proof}

\section{Proof of the general theorem}
The goal of this section is to prove Theorem \ref{theorem: master theorem}. Let $g_n=([n]\times[L],E(g_n))$ be a grid graph such that $G_{g_n}$ converges in the decorated-graphon sense to a decorated graphon $w$. Let $(a_k^{(r)})_{k\in [n],r\in [L]}$ be a family of $g_n$-independent and identically distributed random variable  with common normalized law $a\in \A$. Define
    \begin{align*}
        S_n^{(I)}=\frac{1}{\sqrt{n}}\sum_{k\in [n]}b_n^{(I)};\quad b_n^{(I)}=\prod_{l\in J}a_n^{(l)},
    \end{align*}
    for $I\in \mathcal{P}^*(L)$. We start with the following reduction lemma. 

\begin{lemma}\label{lemma: clt for the joint law}
For any $\alpha\in \{1,*\}^p$ and any $J\in [\mathcal{P}^*(L)]^p$, we have 
    \begin{align*}
        \lim_{n\to \infty}\tau\left((S_n^{(J_1)})^{\alpha_1}\cdots (S_n^{(J_p)})^{\alpha_p}\right)=\sum_{\pi \in P_2(p)}\lim_{n\to \infty}\frac{1}{n^{p/2}}\sum_{\substack{i\in [n]^p\\\ker(i) =\pi}}\tau\left\{\left(b_{i_{1}}^{(J_1)}\right)^{\alpha_1}\cdots \left(b_{i_{p}}^{(J_p)}\right)^{\alpha_p}\right\}.
    \end{align*}
\end{lemma}
\begin{proof}
    We start by writing
    \begin{align*}
        \tau\left((S_n^{(J_1)})^{\alpha_1}\cdots (S_n^{(J_p)})^{\alpha_p}\right)&=\frac{1}{n^{p/2}}\sum_{i\in [n]^p}\tau\left\{\left(b_{i_{1}}^{(J_1)}\right)^{\alpha_1}\cdots \left(b_{i_{p}}^{(J_p)}\right)^{\alpha_p}\right\}\\
        &= \sum_{\pi \in P(p)}\frac{1}{n^{p/2}}\sum_{\substack{i\in [n]^p\\\ker (i) =\pi}}\tau\left\{\left(b_{i_{1}}^{(J_1)}\right)^{\alpha_1}\cdots \left(b_{i_{p}}^{(J_p)}\right)^{\alpha_p}\right\}.
    \end{align*}
    Let $m=|J_1|+\cdots+ |J_p|$. The product
    \begin{align*}
        \left(b_{i_{1}}^{(J_1)}\right)^{\alpha_1}\cdots \left(b_{i_{p}}^{(J_p)}\right)^{\alpha_p}=x_1\cdots x_m
    \end{align*}
    can be written as the product of the elements $x_k=a_{i_j}^{(l)}$, for some $l$ and $j$ that depend only on $\pi,\alpha$ and $J$. Let $\A_{(i,l)}$ be the algebra generated by $a_i^{(l)}$ and denote by $\theta$ the collection of indices $(\theta_k)_{k\in [m]}$ of the algebras associated with the $x_k$'s, namely, $x_k \in \A_{\theta_k}$ for every $k\in [m]$. For instance, if $\pi=\{\{1,2\}\}$, $J_1=\{1,2\}$, $J_2=\{2\}$, $\alpha=(*,1)$ and $\ker(i)=\pi$, then we obtain
    \begin{align*}
&\left(b_{i_{1}}^{(J_1)}\right)^{\alpha_1}\left(b_{i_{2}}^{(J_2)}\right)^{\alpha_2}=a_{i_1}^{(2)}a_{i_1}^{(1)}a_{i_1}^{(2)}=x_1x_2x_3;\\
        &\theta=(\theta_1,\theta_2,\theta_3)=((i_1,2),(i_1,1),(i_1,2)).
    \end{align*}
    Now, suppose $\pi$ contains a block of size one, say $v=\{v\}$. Without loss of generality, assume that $v=1$. Define $\theta_1=(i_1,l_1)$ so that $x_1\in \A_{\theta_1}$. By Speicher's lemma, we have
    \begin{align*}
        \tau(x_1\cdots x_m)=\sum_{\sigma \in \NC(g_n,\theta)}\cumuf_\sigma(x_1,\cdots,x_m).
    \end{align*}
    Since $\theta_1\ne \theta_k$ for all $j\ge 2$ (as $i_1 \ne i_j$ for all $j\ne 1$), and every $\sigma \in \NC(g_n,\theta)$ satisfies $\sigma \le \ker(\theta)$, it follows that $\{1\}\in \sigma$ must be a singleton block. Thus, 
    \begin{align*}
\cumuf_\sigma(x_1,\cdots,x_m)=\cumuf_1(x_1)\cumuf_{\sigma\setminus\{\{1\}\}}(x_2,\cdots,x_m)=0,
    \end{align*}
    since $x_1$ is centered. Therefore, partitions with single blocks do not contribute, leading to
    \begin{align*}
        \tau\left((S_n^{(J_1)})^{\alpha_1}\cdots (S_n^{(J_p)})^{\alpha_p}\right)=\sum_{\pi \in P_{\ge 2}(p)}\frac{1}{n^{p/2}}\sum_{\substack{i\in [n]^p\\\ker (i) =\pi}}\tau\left\{\left(b_{i_{1}}^{(J_1)}\right)^{\alpha_1}\cdots \left(b_{i_{p}}^{(J_p)}\right)^{\alpha_p}\right\}.
    \end{align*}
Suppose $v\in \pi$ satisfies $|v| \ge 3$. Then $|\pi|<p/2$, and we estimate
    \begin{align*}
        \frac{1}{n^{p/2}}\sum_{\substack{i\in [n]^p\\\ker (i) =\pi}}\tau\left\{\left(b_{i_{1}}^{(J_1)}\right)^{\alpha_1}\cdots \left(b_{i_{p}}^{(J_p)}\right)^{\alpha_p}\right\} \le C_{a, J, \alpha}n^{|\pi|-p/2},
    \end{align*}
    where 
    \begin{align*}
        C_{a, J,\alpha}=\max_{i \in [p]^p}\tau\left\{\left(b_{i_{1}}^{(J_1)}\right)^{\alpha_1}\cdots \left(b_{i_{p}}^{(J_p)}\right)^{\alpha_p}\right\}<\infty ,
    \end{align*}
    Since $|\pi|-p/2<0$, it follows that
    \begin{align*}
        \frac{1}{n^{p/2}}\sum_{\substack{i\in [n]^p\\\ker (i) =\pi}}\tau\left\{\left(b_{i_{1}}^{(J_1)}\right)^{\alpha_1}\cdots \left(b_{i_{p}}^{(J_p)}\right)^{\alpha_p}\right\} \to 0.
    \end{align*}
   Thus, in the large-$n$ limit, only pair partitions contribute, leading to 
    \begin{align*}
        \lim_{n\to \infty}\tau\left((S_n^{(J_1)})^{\alpha_1}\cdots (S_n^{(J_p)})^{\alpha_p}\right)=\sum_{\pi \in P_2(p)}\lim_{n\to \infty}\frac{1}{n^{p/2}}\sum_{\substack{i\in [n]^p\\\ker (i) =\pi}}\tau\left\{\left(b_{i_{1}}^{(J_1)}\right)^{\alpha_1}\cdots \left(b_{i_{p}}^{(J_p)}\right)^{\alpha_p}\right\}.
    \end{align*}
This completes the proof. 
\end{proof}
Define
\begin{align*}
    \rho_n(\pi,\alpha,J):=\frac{1}{n^{p/2}}\sum_{\substack{i\in [n]^p\\ \ker(i)=\pi}}\tau\left\{\left(b_{i_{1}}^{(J_1)}\right)^{\alpha_1}\cdots \left(b_{i_{p}}^{(J_p)}\right)^{\alpha_p}\right\}.
\end{align*}
Following the previous proof, let $m=|J_1|+\cdots+ |J_p|$, so that
\begin{align*}
    \left(b_{i_{1}}^{(J_1)}\right)^{\alpha_1}\cdots \left(b_{i_{p}}^{(J_p)}\right)^{\alpha_p}=x_1\cdots x_m,
\end{align*}
where each $x_k=a_{i_j}^{(l)}$ depends only on $\pi,\alpha$ and $J$. Let $\theta_k$ be the algebra associated with $x_k$, i.e., $x_k \in \A_{\theta_k}$. Figure \ref{fig: alpha partition} illustrates the case $\pi=\{\{1,3\},\{2,4\},\{5,6\}\}$, $\alpha=(1,1,*,*,1,1)$, and
\begin{align*}
    &J_1=J_3=\{1,2\};\; J_2=J_4=\{3,4,5\};\;
     J_5=J_6=\{1,2\}.
\end{align*}
\begin{figure}[H]
     \centering
        \begin{subfigure}[b]{0.4\textwidth}
         \centering
         \begin{tikzpicture}[scale=0.50]
            \draw[-] (1,0) edge (1,1);
            \draw[-] (1,1) edge (3,1);
            \draw[-] (3,1) edge (3,0);
        
            \draw[-] (2,0) edge (2,2);
            \draw[-] (2,2) edge (4,2);
            \draw[-] (4,2) edge (4,0);
        
            \draw[-] (5,0) edge (5,1);
            \draw[-] (5,1) edge (6,1);
            \draw[-] (6,1) edge (6,0);
        \end{tikzpicture}
        \caption{Partition $\pi$.}
    \end{subfigure}
     \hfill
     \begin{subfigure}[b]{0.4\textwidth}
         \centering
         \begin{tikzpicture}[scale=0.50]
            \draw[-] (1,0) edge (1,0.7);
            \draw[-] (1,0.7) edge (7,0.7);
            \draw[-] (7,0.7) edge (7,0);
        
            \draw[-] (2,0) edge (2,0.3);
            \draw[-] (2,0.3) edge (6,0.3);
            \draw[-] (6,0.3) edge (6,0);
        
            \draw[-] (3,0) edge (3,2);
            \draw[-] (3,2) edge (10,2);
            \draw[-] (10,2) edge (10,0);

            \draw[-] (4,0) edge (4,1.5);
            \draw[-] (4,1.5) edge (9,1.5);
            \draw[-] (9,1.5) edge (9,0);

            \draw[-] (5,0) edge (5,1);
            \draw[-] (5,1) edge (8,1);
            \draw[-] (8,1) edge (8,0);

            \draw[-] (11,0) edge (11,1);
            \draw[-] (11,1) edge (13,1);
            \draw[-] (13,1) edge (13,0);

            \draw[-] (12,0) edge (12,2);
            \draw[-] (12,2) edge (14,2);
            \draw[-] (14,2) edge (14,0);
        \end{tikzpicture}
        \caption{Partition $\ker(\theta)$.}
     \end{subfigure}
        \caption{Example of $\pi$ and $\ker(\theta)$.}
        \label{fig: alpha partition}
\end{figure}
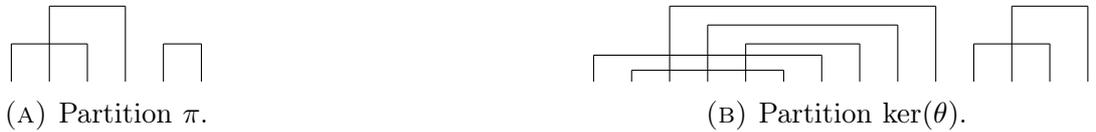
As before, for any partition $\sigma\in P(m)$ containing a singleton block, 
\begin{align*}
    \cumuf_\sigma(x_1,\ldots,x_m)=0. 
\end{align*}
Thus
\begin{align*}
    \tau(x_1\cdots x_m)=\sum_{\sigma \in \NC_{\ge 2}(g_n,\theta)}\cumuf_\sigma(x_1,\cdots,x_m).
\end{align*}
Since $\pi\in P_2(p)$, any $v\in \ker(\theta)$ satisfies $|v| \le 2$, meaning $\ker(\theta)\in P_{\le 2}(m)$. In particular, if $\sigma \in \NC_{\ge 2} (g_n,\theta)$, then $\sigma \le \ker(\theta)$ implies $\sigma \in P_{\le 2}(m)$. Consequently, we must have $\sigma=\ker(\theta)$, leading to $\NC_{\ge 2}(g_n,\theta)=\{\ker(\theta)\}$ whenever it is nonempty.
If $\ker(\theta)\in \NC_2(g_n,\theta)$, it follows that 
\begin{align*}
    \cumuf_{\ker(\theta)}(x_1,\cdots,x_m)=1,
\end{align*}
since the variables are normalized. 
Thus, 
\begin{align*}
    \rho_n(\pi,\alpha,J)=\frac{1}{n^{p/2}}\sum_{\substack{i\in [n]^p\\ \ker(i)=\pi}}\mathbf{1}_{\ker(\theta)\in \NC_2(g_n,\theta)}.
\end{align*}
Next, we give a simple condition under which $\rho_n(\pi,\alpha, J)=0$.
\begin{lemma}\label{lemma: contribution of P_2^J}
    Let $\alpha\in \{1,*\}^p$, $J\in [\mathcal{P}^*(L)]^p$ and $\pi\in P_2(p)$. If there exists $\{r,s\}\in \pi$ such that $J_r \ne J_s$, then $\rho_n(\pi,\alpha,J)=0$.
\end{lemma}
\begin{proof}
Assume that there exists an element $l\in J_{r}\setminus J_{s}$. Without loss of generality, by cyclicity, we take $r=1$ and set $x_1=a_{i_1}^{(l)}$. Since $\pi\in P_2(p)$ is a pair partition, the index $i_{1}$ appears exactly once more, namely as $i_{s}$. Given $l \notin J_{s}$, we observe that $\theta_1=(i_1,l)$ differs from $\theta_k$ for all $k\ge 2$. Consequently, $\{1\}\in \ker(\theta)$, implying that $\ker(\theta) \notin \NC_2(g_n,\theta)$. Thus, $\rho_n(\pi,\alpha,J)=0$, completing the proof.
\end{proof}
Recalling the definition
\begin{align*}
    P_2^J(p)=\{\pi\in P_2(p): \forall \{r,s\}\in \pi, J_r=J_s\}.
\end{align*}
By Lemma \ref{lemma: contribution of P_2^J}, we conclude that  $\rho_n(\pi,\alpha,J)=0$ whenever $\pi\notin P_2^J(p)$. Moreover, $\ker(\theta)\in P_2(m)$ for every $\pi\in P_2^J(p)$.
To complete the argument, we now characterize the condition $\ker(\theta) \in \NC_2(g_n,\theta)$. 
\begin{lemma}\label{lem: characterization of theta partitions}
    Let $\alpha\in \{1,*\}^p$, $J\in [\mathcal{P}^*(L)]^p$, $\pi\in P_2^J(p)$ and $\ker(i)=\pi$. Then $\ker(\theta) \in \NC_2(g_n,\theta)$ if and only if the following conditions are satisfied:
    \begin{enumerate}
        \item\label{condition 1} If a pair $\{r,s\}\in \pi$ satisfies $\alpha_r=\alpha_s$, the subgraph $\{(i_r,l): l\in J_r\}\subseteq g_n$ is complete.
        \item\label{condition 2} If two pairs $\{r_1,s_1\},\{r_2,s_2\}\in \pi$ form a crossing, then $((i_{r_1},l_1),(i_{r_2},l_2))\in E(g_n)$ for all $(l_1,l_2)\in J_{r_1}\times J_{r_2}$.
    \end{enumerate}
\end{lemma}
\begin{proof}
Let us first prove that if the two conditions are satisfied, then $\ker(\theta)\in \NC_2(g_n,\theta)$. Let $v_1,v_2\in \ker(\theta)$ be two crossing blocks. We denote $v_k=\{\theta_1^{(k)},\theta_2^{(k)}\}$, where $\theta_1^{(k)}=(i_{r_k},l_{k})$ and $\theta_2^{(k)}=(i_{s_k},l_{k})$, for $k=1,2$ ($i_{r_k}=i_{s_k}$). We either have i) $i_{r_1}=i_{r_2}$; or ii) $i_{r_1}\ne i_{r_2}$. 
\begin{enumerate}
    \item Case (i) ensures that $\{r_1,r_2\}\in \pi$, and by definition of $\ker(\theta)$, we have $\alpha_{r_1}=\alpha_{r_2}$.
    \item Case (ii) ensures that $\{r_1,s_1\},\{r_2,s_2\}\in \pi$ form a crossing, by definition of $\ker(\theta)$.
\end{enumerate}
In both cases, Conditions \ref{condition 1} and \ref{condition 2} imply that $((u_{r_1},l_1),(u_{r_2},l_2))\in E(g_n)$ and it follows that $\ker(\theta)\in \NC_2(g_n,\theta)$.

Now suppose that $\ker(\theta)\in \NC_2(g_n,\theta)$. Let us prove that the two conditions are satisfied.
\begin{enumerate}
    \item Let $\{r,s\}\in \pi$ be such that $\alpha_r=\alpha_s$. By definition of $\ker(\theta)$, the blocks $v_1=\{(i_r,l_1),(i_s,l_1)\}$ and $v_2=\{(i_r,l_2),(i_s,l_2)\}$ will cross for any distinct $l_1,l_2\in J_r$. Since $\ker(\theta)\in \NC_2(g_n,\theta)$, it follows that $((i_r,l_1),(i_r,l_2))\in E(g_n)$ for any distinct $l_1,l_2\in J_r$. The first condition follows.
    \item Let $\{r_1,s_1\},\{r_2,s_2\}\in \pi$ be such that they form a crossing. By definition of $\ker(\theta)$, the blocks $v_1=\{(i_{r_1},l_1),(i_{s_1},l_1)\}$ and $v_2=\{(i_{r_2},l_2),(i_{s_2},l_2)\}$ will cross for any $(l_1,l_2)\in J_{r_1}\times J_{r_2}$. Since $\ker(\theta)\in \NC_2(g_n,\theta)$, it follows that $((i_{r_1},l_1),(i_{r_2},l_2))\in E(g_n)$ for any $(l_1,l_2)\in J_{r_1}\times J_{r_2}$. The second condition follows, and the lemma is proved.
\end{enumerate}
\end{proof}
We are now ready to finish the proof of Theorem~\ref{theorem: master theorem}. In view of Lemma~\ref{lem: characterization of theta partitions}, we deduce that for any $\pi\in P_2^J(p)$, we have
\begin{align*}
    \rho_n(\pi,J,\alpha)&=\frac{1}{n^{p/2}}\sum_{\substack{i\in [n]^p\\ \ker(i)=\pi}}\prod_{\substack{\{r,s\}\in V(f_\pi)\\ \alpha_r=\alpha_s}}\left(\mathbf{1}_{\{(i_r,l): l\in J_r\}\subseteq g_n \text{ is complete}}\right) \\
    &\times\prod_{(\{r_1,s_1\},\{r_2,s_2\})\in E(f_\pi)}\left(\prod_{(l_1,l_2)\in J_{r_1}\times J_{r_2}}\mathbf{1}_{(i_{r_1},l_1),(i_{r_2},l_2))\in E(g_n)}\right).
\end{align*}
Recalling the definitions of $F_{f_\pi}^{J,\alpha}=(\tilde{f}_{\pi},\beta)$ and $G_{g_n}=(K_n,\gamma_n)$, and in view of the one-to-one map from injective maps $\phi:V(f_\pi) \hookrightarrow [n]$ to $i\in [n]^p$ such that $\ker(i)=\pi$ given by $\phi(u)=i_r$ for a block $u=\{r,s\}\in \pi$, Lemma \ref{lemma: rho computation} implies that
\begin{align*}
    \rho_n(\pi,J,\alpha)=\left(\frac{n}{n-p/2}\right)\rho_{\operatorname{inj}}(F^{J,\alpha}_{f_\pi}, G_{g_n}).
\end{align*}

Since $\rho_{\operatorname{inj}}(F_{f_\pi}^{J,\alpha},G_{g_n})$ and $\rho(F_{f_\pi}^{J,\alpha},G_{g_n})$ are asymptotically equivalent as $n\to\infty$ \cite[Section 2]{lovasz2022multigraph}, we conclude that
\begin{align*}
    \lim_{n\to \infty}\tau\left((S_n^{(J_1)})^{\alpha_1}\cdots (S_n^{(J_p)})^{\alpha_p}\right)=\sum_{\pi\in P_2^J(p)}\rho(F_{f_\pi}^{J,\alpha},w).
\end{align*}

\bibliographystyle{abbrvnat}

\end{document}